\documentclass[a4paper]{article}
 \usepackage{amsmath, amssymb,amscd} 
 \usepackage{enumerate} 
 \usepackage[active]{srcltx} 
 \usepackage{amsthm}

\usepackage{natbib} 
\usepackage{amsthm} 
\usepackage{xargs}

\theoremstyle{definition}
\newtheorem{definition}{Definition}

\newtheorem{example}{Example}

\theoremstyle{plain}
\newtheorem{theorem}{Theorem}
\newtheorem{lemma}{Lemma}
\newtheorem{proposition}{Proposition}
\newtheorem{corollary}{Corollary}

\def\RR{\mathbb{R}}
\def\NN{\mathbb{N}}
\def\states{\mathcal X}
\def\allgambles{\mathcal L}
\def\comm{\leftrightsquigarrow}
\def\supp{\textup{supp}\,}
\def\essmax{\textup{ess\,max}\,}

\newcommandx{\acs}[1][1=]{\stackrel{#1}{\rightsquigarrow}} 
\newcommandx{\sacs}[1][1=]{\stackrel{#1}{\rightarrow}} 
\newcommandx{\up}[1]{\overline{#1}} 
\newcommandx{\low}[1]{\underline{#1}} 
\newcommandx{\chf}[1]{1_{#1}} 
\newcommandx{\state}[1][1=x]{#1} 
\newcommandx{\cclass}[1][1=D]{\mathcal{#1}} 
\newcommandx{\sset}[1][1=A]{#1} 
\newcommandx{\ifnal}[1][1=E]{#1} 
\newcommandx{\fnal}[1][1=P]{#1} 
\newcommandx{\toper}[1][1=T]{#1} 
\newcommandx{\ptoper}[1][1=t]{#1} 
\newcommandx{\uifnal}[1][1=E]{\up{\ifnal[#1]}} 
\newcommandx{\utoper}[1][1=T]{\up{\toper[#1]}} 
\newcommandx{\lifnal}[1][1=E]{\low{\ifnal[#1]}} 
\newcommandx{\ltoper}[1][1=T]{\low{\toper[#1]}} 
\newcommandx{\psets}[1][1=\ifnal]{\psi_{{#1}}} 
\newcommandx{\acsets}[1][1=\toper]{\Theta_{{#1}}} 
\newcommandx{\gamble}[1][1=f]{#1} 

\def\statey{\state[y]}
\def\statex{\state[x]}

\begin{document}
\author{Damjan \v{S}kulj}
\title{A classification of invariant distributions and convergence of imprecise Markov chains}

\maketitle
\begin{abstract}
We analyse the structure of imprecise Markov chains and study their convergence by means of accessibility relations. We first identify the sets of states, so-called \emph{minimal permanent classes}, that are the minimal sets capable of containing and preserving the whole probability mass of the chain. These classes generalise the \emph{essential classes} known from the classical theory. We then define a class of extremal imprecise invariant distributions and show that they are uniquely determined by the values of the upper probability on minimal permanent classes. Moreover, we give conditions for unique convergence to these extremal invariant distributions. 


\medskip \noindent
{\bfseries Math. subj. class. (2000)} 60J10, 93B35,  47N30
\end{abstract}

\section{Introduction}
The term \emph{imprecise Markov chain} denotes a Markov chain whose parameters are not fully determined, but rather some partial information about them is given in terms of models of \emph{imprecise probabilities}. The systematic modelling of Markov chains with partially determined parameters was initiated by Hartfiel and Seneta (see e.g. \citep{hart:94, hart:98} and the references therein) under the name 'Markov set chains'. Their theory is not formally built on the models of imprecise probabilities, such as Walley's model (\citep{walley:91}), although they develop similar concepts and methods. More formally the models of imprecise probabilities have been involved in the study of Markov chains by the approaches proposed by \citet{skulj:09}, where general models of interval probabilities (see e.g. \citep{weic:01}) are involved, and \citet*{decooman-2008-a}, who introduce the approach with upper and lower expectation functionals. 

One of the main considerations of all mentioned approaches are the conditions for unique convergence to an invariant (imprecise) limit distribution. It has been shown that under certain assumptions imprecise Markov chains behave in a similar way as the precise ones, allowing the possibility to generalise concepts, such as accessibility relations \citep{decooman-2008-a}, coefficients of ergodicity \citep{hart:98, skulj:hable:12}, and, for the chains with certain absorption, the invariant distributions conditional on non-absorption \citep{crossman:2010}.

The main topic of the present paper are the imprecise Markov chains that do not converge uniquely, which means that the limit imprecise distribution depends on the initial state or distribution over the set of states. In the classical theory this case is relatively easily translated to the case with unique convergence by partitioning the set of states into disjoint \emph{communication classes}, which are then either \emph{essential} or \emph{transient}. Essential classes have the property that the probability mass concentrated within them and in the transient classes leading to them is preserved within the class. In the long term it is then either distributed among states in a stationary way or moves between the states in \emph{cycles}. Transient classes, though, in the long term transfer all the probability mass to the essential classes. The long term behaviour of a precise Markov chain is then described with the proportions of the probability mass within each essential class, that can then be examined, each separately without influencing the others, as either ergodic or cyclic single communication classes. (For more details see e.g. \citep{Seneta2006}.)

The situation in the case of imprecise Markov chains is considerably more complex. In the precise case the probability mass can only move from one state into another, which is determined by strictly positive transition probability. In the imprecise case, though, it can happen that instead we have a probability interval that has a strictly positive upper and zero lower bound. Therefore, we have to take into account both scenarios -- where a transition is possible and where it is not -- within the same model (for better clarification see Example~\ref{ex-1}). This leads to numerous interesting new possibilities that are not observed in the precise theory. One immediate consequence is that there are different possible ways to define communication relations between states. \emph{Say for instance, when the upper probability of transition from a state $\statex$ to another state $\statey$ is positive and the lower one zero, can we say that $\statey$ is accessible from $\statex$ or not?} We argue that questions of this kind cannot be answered with a single accessibility relation. Therefore, in addition to the accessibility relation defined by \citet{decooman-2008-a}, that is based on upper probabilities and which we describe in Subsection~\ref{ss-war}, we define another accessibility relation in Subsection~\ref{ss-sar}. This relation is defined between sets of states rather than between single states. Combining both relations will allow us to define minimal classes of states, called \emph{minimal permanent classes}, that generalise essential communication classes from the classical theory in the sense that they  are the minimal sets of states capable of possibly preserving the entire probability mass concentrated within them. However, unlike in the case of essential classes in the classical theory, this property does not prevent the possibility that the mass that can possibly be preserved within one class is entirely or partially transferred elsewhere, that is to other minimal permanent classes. Moreover, these classes are neither necessarily disjoint. In effect, a complex network of interdependence between the classes in possible in general. Our first main result in Theorem~\ref{thm-main-1} shows that nevertheless all invariant imprecise distributions that possibly concentrate the entire probability mass in one or more minimal permanent classes are uniquely determined by the set of such classes. Further, our second main result in Theorem~\ref{thm-main-2} shows that when probability mass is at least in the limit concentrated in these classes the imprecise distributions over the set of states converge to the unique invariant distributions. These two theorems present the main results of this paper. 

In addition to the main theorems, we believe that the analysis of the pair of accessibility relations contributes to better understanding of imprecise Markov chains in general, describing the complexity of their structure, and on the other hand, the induced minimal permanent classes seem to present the smallest irreducible units, counterparts to the essential communication classes in the classical Markov chains. Although, considerably more apparently irreducible complexity remains in the imprecise case due to complex interdependencies that may be present between the classes. 

The paper has the following structure. In the next section we formally describe the model of imprecise Markov chains. In Section~\ref{s-arc} we first describe the accessibility relation based on upper transitions, previously defined by \citet{decooman-2008-a}, which here is referred as the weak accessibility relation, and state the unique convergence theorem, that was one of the main results in the above reference. The so-called strong accessibility relation is then defined in Subsection~\ref{ss-sar}, which allows the definition of the so-called permanent classes in Subsection~\ref{ss-pc}. In Section~\ref{s-cuc} we finally use these results to prove our main results in Theorem~\ref{thm-main-1} and Theorem~\ref{thm-main-2}. 

\section{Imprecise Markov chains}
%
%

Let $\states$ be a non-empty finite set of \emph{states} and $\{ X_n \}_{n\in \NN}$ a sequence of random variables taking values in $\states$. The probability distributions of the variables are assumed to be imprecise. Throughout the paper an \emph{imprecise probability (distribution)} denotes a closed convex set of expectation functionals\footnote{The notation here is very simplified version of the notation used by \citet{walley:91}, which is still enough general and clear for our purpose.} $\ifnal$. The elements of this set will be denoted by $\fnal$. 
Sets of expectation functionals are of course equivalent to sets of corresponding probability distributions. Instead of the term imprecise probability distribution, we will call $\ifnal$ an \emph{imprecise expectation functional (IEF)}\footnote{Imprecise expectation functionals correspond to \emph{lower} and \emph{upper previsions} which constitute one of the key objects studied in the theory of imprecise probabilities. The way how imprecise expectation functionals are used here corresponds to \emph{credal sets} of lower and upper previsions.}. 

In our case where the set of states is finite every IEF is a compact set. A convenient way to represent IEFs is via their lower and upper expectation functionals as follows. Let $\ifnal$ be an IEF and $\gamble$ a real valued map $\states\to \RR$. Then the \emph{lower} and the \emph{upper expectation} with respect to $\ifnal$ are 
\begin{align*}
 \lifnal (\gamble) = \min_{\fnal\in\ifnal} \fnal(\gamble) \\
 \intertext{and}
 \uifnal (\gamble) = \max_{\fnal\in\ifnal} \fnal(\gamble).
\end{align*}
We will call real valued maps $\gamble\colon \states\to\RR$ \emph{gambles}. The set of all gambles on $\states$ is denoted by $\allgambles(\states)$, or simply by $\allgambles$. The values of $\fnal(\gamble)$ for a given gamble $\gamble$ and all functionals $\fnal\in\ifnal$ therefore form the interval $[\lifnal(\gamble), \uifnal(\gamble)]$, and we may therefore identify $\ifnal(\gamble)$ with this interval. In particular, when $\gamble = \chf{\sset}$, where $\sset$ is some set of states, the value $\fnal(\chf{\sset})$ denotes the probability of $\sset$, and $\lifnal (\chf{\sset})$ and $\uifnal (\chf{\sset})$ are the \emph{lower} and the \emph{upper probability} of $\sset$. 

The set of IEFs is ordered with respect to set inclusion, and in the case where $\ifnal \subseteq \ifnal[F]$ we say that $\ifnal$ is \emph{more committal}\footnote{This expression is a consequence of the fact that a more committal imprecise expectation functional is 'less imprecise', which in Walley's theory commits to accepting more gambles.} than $\ifnal[F]$. Clearly, an IEF $\ifnal$ is more committal than $\ifnal[F]$ if and only if $\lifnal(\gamble) \ge \lifnal[F](\gamble)$, or equivalently, $\uifnal(\gamble) \le \uifnal[F](\gamble)$, for every gamble $\gamble$. 

The subset of gambles mapping $\states$ to the interval $[0, 1]$ will be denoted by $\allgambles_1$, and it is a compact set. It can be easily verified that the values of imprecise expectation functionals on $\allgambles_1$ uniquely determine their values on the whole $\allgambles$. Particularly, $\uifnal[E](f) \le \uifnal[F](f)$ for every gamble $f\in \allgambles_1$ implies $\ifnal[E]\subseteq \ifnal[F]$.

The least committal imprecise expectation functional on $\allgambles(\states)$ is the \emph{vacuous} IEF which contains all expectation functionals. The vacuous IEF will usually be denoted by $\ifnal[V]$. For every gamble  $\gamble\in\allgambles(\states)$ we have 
\[ \lifnal[V](\gamble) = \min_{\state\in\states}\gamble(\state) \qquad\text{and}\qquad \uifnal[V](\gamble) = \max_{\state\in\states}\gamble(\state). \]

In the sequel we will list some basic properties of lower and upper expectation functionals. First we have that 
\begin{equation*}
 \lifnal(\gamble) = -\uifnal(-\gamble). 
\end{equation*}
Al the properties of lower expectation functionals are therefore easily derived from the corresponding properties of the upper expectation functionals, which are the following. Let $\ifnal$ be an IEF, $\gamble, \gamble_1, \gamble_2$ arbitrary gambles, $\lambda$ a non-negative real constant, $\mu$ an arbitrary real constant, and $\chf{\sset}$ the characteristic function of a set $\sset\subseteq \states$. Then:
\begin{enumerate}[(i)]
 \item $\min_{\state\in\states} \gamble(\state) \le \uifnal(\gamble)\le \max_{\state\in\states} \gamble(\state)$ (boundedness);
 \item $\uifnal(\gamble_1+\gamble_2) \le \uifnal(\gamble_1) + \uifnal(\gamble_2)$ (subadditivity);
 \item $\uifnal(\lambda\gamble) = \lambda\uifnal(\gamble)$ (non-negative homogeneity);
 \item $\uifnal(\gamble + \mu\chf{\states}) = \uifnal(\gamble) + \mu$ (constant additivity);
 \item if $\gamble_1\le \gamble_2$ then $\uifnal(\gamble_1) \le \uifnal(\gamble_2)$ (monotonicity). 
\end{enumerate}
Conversely, every functional satisfying conditions (i)-(iii) above is an upper expectation functional for some IEF.

We will often need the following simple result. 
\begin{lemma}\label{x-krat-char}
Let $\fnal$ be an expectation functional, $f$ a gamble and $a\ge 0$ a constant. Then 
\begin{enumerate}[{(i)}]
 \item if $f\ge 0$ then $\fnal(\gamble) \ge a\fnal(\chf{\{ \gamble\ge a \}});$
 \item $\fnal(\gamble) = \bar\gamble$ if and only if $\fnal(\chf{\{ \gamble = \bar\gamble \}}) = 1$, 
\end{enumerate}
where $\bar\gamble = \max_{\state\in\states}\gamble(\state)$. The above inequalities are also valid for $\lifnal$ and $\uifnal$, where $\ifnal$ is an IEF. 
\end{lemma}
\begin{proof}
To see (i) note that $a\chf{\{ \gamble\ge a \}}\le \gamble$ and $ a\fnal(\chf{\{ \gamble\ge a \}}) = \fnal(a\chf{\{ \gamble\ge a \}}) \le  \fnal(\gamble)$ because of monotonicity of $P$. (ii) is obvious.
\qed
\qed
\end{proof}
We continue with the description of imprecise Markov chains. The imprecise probability distribution corresponding to the random variable $X_n$ is described by an IEF $\ifnal_n$. Given the distribution of $X_n$, the distribution of $X_{n+1}$ (and consequently of the variables corresponding to later steps) can be found via \emph{transition probabilities}:
\begin{equation}\label{t-prob}
 P(X_{n+1} =\cdot | X_{n} = \statex). 
\end{equation}
When the above transition probability is imprecise, it will be denoted with an IEF $\toper(\cdot | \state)$. Given a gamble $\gamble$ the conditional IEF $\toper(\cdot | \state)$ maps it to the set $\toper(\gamble | \state)$. With a fixed gamble the latter expression is an interval mapping on the set of states $\states$, denoted by $\toper \gamble$. Every $\toper \gamble(\state) = \toper(\gamble|\state)$ denotes a closed interval $[\ltoper\gamble(\state), \utoper\gamble(\state)] = [\ltoper(\gamble|\state), \utoper(\gamble|\state)]$. Hence, $\toper\gamble$ is the set of gambles $\gamble[h]$ satisfying $\ltoper\gamble(\state) \le \gamble[h](\state)\le \utoper\gamble(\state)$. The so defined operator $\toper$ is called an \emph{imprecise transition operator (ITO)}.

An ITO is a closed convex set of (precise) transition operators; however, not every such set is an ITO. (We will denote precise transition operators with small letters, usually $\ptoper$.) It has to be additionally required that a closed convex set of transition operators  has \emph{separately specified rows} to be an ITO, which means that if the transition operators $\ptoper$ and $\ptoper'$ belong to $\toper$ then for every $\state\in\states$ the operator $\ptoper''$ such that 
\begin{equation*}
 \ptoper'' \gamble (\statey) = \begin{cases}
                               \ptoper \gamble (\statey) & \text{if } \statey \neq \state \\
                               \ptoper' \gamble (\statey) & \text{if } \statey = \state
                              \end{cases}
\end{equation*}
also belongs to $\toper$. It is easy to verify that separately specified rows of an ITO ensure that there is some $\ptoper\in\toper$ for every $\gamble[h]\in \toper\gamble$ so that $\gamble[h]=\ptoper \gamble$, and in particular, that $\ltoper \gamble$ and $\utoper \gamble$ belong to $\toper \gamble$. 

In order to allow calculations with imprecise transition operators, we must extend the definition of $\toper \gamble$ from gambles to intervals of gambles of the form 
\[ [\underline{\gamble}, \overline{\gamble}] := \{ \gamble[h]\colon \underline{\gamble}\le\gamble[h]\le \overline{\gamble} \} . \] 
We define
\[ \toper{} [\underline{\gamble}, \overline{\gamble}] = [\ltoper\underline{\gamble}, \utoper\,\overline{\gamble}],  \] 
which clearly contains exactly all the gambles of the form $\ptoper \gamble$ where $\ptoper\in\toper$ and $\gamble\in [\underline{\gamble}, \overline{\gamble}]$. 

Let an imprecise Markov chain be given in terms of an imprecise initial distribution in the form of IEF $\ifnal_0$, and an ITO $\toper$. The imprecise distribution at time $n$, given in the form of IEF $\ifnal_n$, is calculated with
\begin{equation}\label{mc-calc}
 \ifnal_n (\gamble) = \ifnal_0 (\toper^n\gamble), 
\end{equation}
where $\toper^n \gamble$ is given recursively with $\toper^0\gamble = \gamble$ and $\toper^{k+1}\gamble = \toper{} [\toper^k\gamble]$ for every $k\ge 0$.  

The following holds.
\begin{proposition}\label{convex-ief}
 Let IEFs $\ifnal$ and $\ifnal[F]$ be given. Then 
 \[ (\alpha \ifnal + \beta \ifnal[F])\toper = \alpha \ifnal\toper + \beta \ifnal[F]\toper \] 
 if $\alpha, \beta\ge 0$ and $\alpha + \beta = 1$. 
\end{proposition}
\begin{proof}
 Take any gamble $\gamble$, and calculate
 \begin{align*}
  (\alpha \ifnal + \beta \ifnal[F])(\toper \gamble) & = (\alpha \ifnal + \beta \ifnal[F])([\ltoper \gamble, \utoper \gamble]) \\
  & = [(\alpha \lifnal + \beta \lifnal[F])(\ltoper \gamble), (\alpha \uifnal + \beta \uifnal[F])(\utoper \gamble)] \\
  & = [\alpha \lifnal(\ltoper \gamble) + \beta \lifnal[F](\ltoper \gamble), \alpha \uifnal(\utoper \gamble) + \beta \uifnal[F](\utoper \gamble)] \\
  & = \alpha [\lifnal(\ltoper \gamble), \uifnal(\utoper \gamble)] + \beta [\lifnal[F](\ltoper \gamble), \uifnal[F](\utoper \gamble)] \\
  & = (\alpha \ifnal\toper + \beta \ifnal[F]\toper)(\gamble).
 \end{align*}
\qed
\end{proof}
Much of the attention in this paper is put to invariant imprecise expectation functionals, which generalise the concept of invariant distributions. 
Let $\toper$ be an ITO. An IEF $\ifnal$ is called \emph{$\toper$-invariant} if $\ifnal \toper = \ifnal$. Clearly, if $\ifnal$ is $\toper$-invariant, then it is also $\toper^n$-invariant, for every $n\in \NN$.
 
It has been shown previously in \citep{skulj:09} that there always exists the least committal $\toper$-invariant IEF for every ITO $\toper$, and under certain assumptions also the most committal one. We have the following proposition. 
\begin{proposition}\label{limit-invar}
 Let an ITO $\toper$ be given and an IEF $\ifnal_0$ be such that either 
 \begin{align*}
 	 \ifnal_0 \toper &\subseteq \ifnal_0
 	 \intertext{or}
 	 \ifnal_0 \toper &\supseteq \ifnal_0.
 \end{align*}
Then the sequence $\{ \ifnal_n \}_{n\in \NN}$, where $\ifnal_{n+1} = \ifnal_n \toper$, is monotone and the limit $\ifnal_\infty = \lim_{n\to\infty}E_n$ exists and is a $\toper$-invariant IEF. 
\end{proposition}
It is an immediate consequence of this proposition that taking $\ifnal_0$ to be the vacuous IEF on the set of states $\states$, the limit IEF is the unique least committal $\toper$-invariant IEF, but in general there is no unique most committal $\toper$-invariant IEF. 
\begin{definition}
 Let $\ifnal$ be an IEF. Then we define the \emph{support} of $\ifnal$ with 
 \[ \supp(\ifnal) = \{ \state\colon \uifnal(\chf{\{\state\}})>0 \}. \]
\end{definition}
The following proposition is straightforward.
\begin{proposition}\label{gamble-support}
For every IEF $\ifnal$ and every gamble $\gamble \in \allgambles(\states)$ we have that
 \[ \ifnal (\gamble) = \ifnal (\gamble\cdot \chf{C}), \]
 whenever ${\supp(\ifnal)}\subseteq C$.
\end{proposition}

\section{Accessibility relations and convergence}\label{s-arc}
In the theory of Markov chains it is an important question which states are possible to be visited at following times when a state has been observed at given time. In the case of precise Markov chains there is a clear meaning of the expression 'possible', which states that the probability is strictly positive. While in the case of imprecise Markov chains the situation is more complex. The probabilities are now in the form of intervals, and these intervals can have positive upper bound but zero lower bound. This suggests that there may be different notions of accessibility relations, reflecting different notions of 'possibility'. 

We first describe the notion of accessibility introduced by \citet{decooman-2008-a}, based on upper probability. This accessibility notion is based on proclaiming an event possible if it has strictly positive upper probability. Note that the probability (interval) of the chain being in a state $\statey$ at time $n$ given that it is in the state $\state$ at time 0 is obtained using formula \eqref{mc-calc} as $\toper^n\chf{\{\statey\}}(\statex)$.  

\subsection{Weak accessibility relation and unique convergence}\label{ss-war}
The accessibility relation based on the upper probabilities is defined as follows. A state $\statey$ is \emph{accessible} from $\statex$ in $n$ steps if $\utoper^n\chf{\{\statey\}}(\statex)>0$. We will then write $\statex\acs[n] \statey$. When $\statex \acs[n] \statey$ for some $n\ge 0$, we write $\statex \acs \statey$. If both $\statex \acs \statey$ and $\statey \acs \statex$ then we say that the states \emph{communicate} and write $\statex \comm\statey$. The relation $\acs$ is a preorder, i.e. is reflexive and transitive, and $\comm$ is an equivalence relation that partitions the set of states $\states$ into \emph{communication classes}. Every two elements of a communication class are then accessible from one another. 
\begin{definition}
 A set $\sset[C]\subseteq \states$ of states is \emph{absorbing} if $\statex \in \sset[C]$ and $\statex\acs \statey$ implies $\statey\in\sset[C]$. 
\end{definition}
\begin{proposition}\label{abs-monotone}
 Let $\sset[C]$ be absorbing. Then 
 \begin{equation}\label{eq-vac-mon}
	\ifnal[V]_{C}\toper\subseteq \ifnal[V]_{C}
 \end{equation} 
 where $\ifnal[V]_{\sset[C]}$ is the vacuous IEF on $\sset[C]$. 
\end{proposition}
\begin{proof}
 Clearly we have that $\supp(\toper(\cdot|\state))\subseteq \sset[C]$ for every $\state\in \sset[C]$, and therefore, by Proposition~\ref{gamble-support}, $\utoper \gamble (\state) = \utoper(\gamble|\state) = \utoper(\gamble 1_C|\state) \le \uifnal[V]_{\sset[C]}(\gamble)$. Hence, $\uifnal[V]_{\sset[C]}(\utoper \gamble) \le \uifnal[V]_{\sset[C]}(\gamble)$, which implies \eqref{eq-vac-mon}. 
\qed
\end{proof}
The following corollary follows immediately from Propositions~\ref{limit-invar} and \ref{abs-monotone}. 
\begin{corollary}\label{absorbing-invariant-existence}
 Let $\sset[C]$ be absorbing. Then there exists the unique least committal $\toper$-invariant IEF $\ifnal$ such that $\supp(\ifnal)\subseteq \sset[C]$. 
\end{corollary}
\begin{proposition}\label{invariant-absorbing}
 Let $\ifnal$ be a $\toper$-invariant IEF. Then $\supp(\ifnal)$ is absorbing. 
\end{proposition}
\begin{proof}
 Let $\state\in\supp(\ifnal)$ and $\state\acs\statey$. By definition then there is some $r>0$ such that $\utoper^r \chf{\{ \statey\}}(\state)>0$. $\toper$-invariance of $\ifnal$ implies 
 \[ \uifnal(\chf{\{ \statey\}}) = \uifnal(\utoper^r\chf{\{ \statey\}}) \ge \uifnal(\chf{\{ \state\}})\utoper^r\chf{\{ \statey\}}(x) > 0. \] 
 Hence $\statey\in \supp(\ifnal)$ as well. 
\qed
\end{proof}
Take two communication classes $\cclass[C]$ and $\cclass[D]$. If for some states $\statex\in\cclass[C]$ and $\statey\in\cclass[D], \statex$ leads to $\statey$, and hence every element in $\cclass[C]$ leads to every element in $\cclass[D]$, then we say that the class $\cclass[C]$ leads to $\cclass[D]$. If a class $\cclass[D]$ only leads to itself, then we call it a \emph{maximal class}. When there is a unique maximal class it is called the \emph{top class}. A communication class is called \emph{regular} if for every pair of its elements $\statex$ and $\statey$ there is some $r$ such that $\statex \acs[n]\statey$ for every $n\ge r$. If the top class is regular, then the accessibility relation is said to be \emph{top class regular}. An imprecise Markov chain is called \emph{regularly absorbing} if it is top class regular and if there is some $n\in \NN$ such that 
\begin{equation}\label{reg-abs}
\ltoper^n 1_{R}(\statey)>0 \text{ for every } \statey\in \states\backslash R, 
\end{equation}
where $R$ is the top class. In other words, the lower probability that the chain will eventually move from any class outside the top class to the top class is positive. This practically means that no matter what is the initial distribution of the probability mass, it gets transferred to the top class, making it an absorbing class. So in the long run all probability mass is moved into the top class. 

The following convergence theorem is one of the most general results on unique convergence for imprecise Markov chains. 
\begin{theorem}[\citep*{decooman-2008-a}, Theorem~5.1]\label{conv-decooman}
 Consider a stationary imprecise Markov chain with finite set $\states$ that is regularly absorbing. Then for every initial upper expectation $\uifnal_0$ the upper expectation $\uifnal_n = \uifnal_0 \utoper^n$ for the state at time $n$ converges point-wise to the same upper expectation $\uifnal_\infty$:
 \[ \lim_{n\to\infty} \uifnal_n (\gamble) = \lim_{n\to\infty} \uifnal_0 (\toper^n \gamble) =: \uifnal_\infty(\gamble) \text{ for every }\gamble \in \allgambles(\states). \]
 Moreover, the limit expectation $\uifnal_\infty$ is the only $\utoper$-invariant upper expectation on $\allgambles(\states)$. 
\end{theorem}
The main objective of this paper is to explore the situations where the assumptions of the above theorem are not satisfied. Let us first give a motivating example with a top class regular imprecise Markov chain that is not regularly absorbing, i.e. does not satisfy \eqref{reg-abs}. 
\begin{example}\label{ex-1}
 We consider an imprecise Markov chain with the set of states $\states$ partitioned in two subsets $\states_1$ and $\states_2$ and ITO $T$ that can be represented in the block matrix form:
 \[ \toper = \begin{bmatrix}
         P & \mathbf 0 \\
         R & Q
        \end{bmatrix}, \]
where the blocks correspond to $\states_1$ and $\states_2$ respectively, so that $P$ is a regular imprecise operator, $\utoper[R]>0$ and $\ltoper[R]=0$. Clearly, the corresponding chain is top class regular, where $\states_1$ is its top class; however, it is not regularly absorbing because of $\ltoper[R]=0$. 

Let us first show that indeed the chain is not uniquely convergent. In fact we will construct invariant IEFs in three different ways, two of them being clearly different, while the equality between the second and the third one is not obvious and will turn out as a consequence of Theorem~\ref{thm-main-1}. 

The first $\toper$-invariant IEF is the one where the probability mass is concentrated in $\states_1$, which is an absorbing set. By Theorem~\ref{conv-decooman} starting with any IEF of the block form $(\ifnal^1_0, 0)$ the IEFs $(\ifnal^1_n, 0) = (\ifnal^1_0, 0)\toper^n$, converge uniquely to a $\toper$-invariant $E_\infty = (\ifnal^1_\infty, 0)$. Where $\ifnal^1_\infty$ is the unique $P$-invariant IEF.

Now take $\ifnal[V]$ to be the vacuous IEF on $\states$. Since $\ltoper[R]=0$ implies $\utoper[Q]=1$, this implies that $\uifnal[V](\utoper^n1_{\states_2}) = 1$, as well as $\uifnal[V](\utoper^n1_{\states_1}) = 1$ because of $\utoper[P]=1$. But on the other hand the sequence $\{ \ifnal[V]\toper^n \}$ converges to some $\ifnal[F]_\infty$, which must then satisfy both $\uifnal[F]_\infty(1_{\states_1}) = 1$ and $\uifnal[F]_\infty(1_{\states_2}) = 1$. 

Yet another way to find a $\toper$-invariant IEF is the following. Let $\ifnal[F]' = (0, \ifnal[V]_2)$ where $\ifnal[V]_2$ is the vacuous IEF on $\states_2$. Consider an ITO of the form 
\[ \toper_1 = \begin{bmatrix}
         P & \mathbf 0 \\
         \mathbf 0 & Q^*
        \end{bmatrix} \subseteq \toper. \]
For simplicity we may assume that $Q^*$ is regular, which implies that $\states_2$ is a communication class. The sequence $\ifnal[F]'_n = \ifnal[F]'\toper_1^n$ clearly satisfies the conditions of the Proposition~\ref{limit-invar} and therefore there exists the limit $\ifnal[F]'_\infty = \lim_{n\to\infty}\ifnal[F]'_n = (0, \ifnal[F]^2_\infty)$. Assuming regularity of $Q^*$ it is then the unique $\toper_1$-invariant IEF of this form. But then $\ifnal[F]'_\infty \toper \supseteq \ifnal[F]'_\infty$, which by Proposition~\ref{limit-invar} again implies that the limit $\ifnal[\tilde F]_\infty
= \lim_{n\to\infty} \ifnal[F]'_\infty \toper^n$ exists and is $\toper$-invariant. Moreover, $\utoper[R]>0$ implies that $\uifnal[{\tilde F}]_\infty(1_{\states_1})=\uifnal[\tilde F]_\infty(1_{\states_2})=1$. Thus the values of $\uifnal[\tilde F]_\infty$ and $\uifnal[F]_\infty$ coincide on the characteristic functions of both communication classes in $\states$. It is less obvious though that this implies that they must be equal, as follows from our first main result in Theorem~\ref{thm-main-1}.
        
Although $\ifnal_\infty$ and $\ifnal[F]_\infty$ are two extremal invariant IEFs, many more can be constructed simply by taking their convex combinations, as easily follows from Proposition~\ref{convex-ief}. It is still an open question, however, whether every non-extremal invariant IEF can be constructed as a convex combination of the extremal ones. 
\end{example}

\subsection{Strong accessibility relation}\label{ss-sar}
The accessibility relation $\acs$ denotes the possibility that one element is accessible from another. Although, it might be that the lower probability of an $\state$ leading to a $\statey$ is zero and the upper one is positive, which clearly can have essentially different implications than the case where both lower and upper probability are positive. Example~\ref{ex-1} describes such a case. The weak accessibility relation $\acs$ is therefore not enough to describe all the relevant properties of the behaviour of an imprecise Markov chain. 

For this reason we will define another accessibility relation between sets of states as follows.
\begin{definition}
Let $\sset$ and $\sset[B]$ be arbitrary sets of states. We say that $\sset$ \emph{strongly leads to $\sset[B]$ in $n$ steps} if $\utoper^n\chf{B}(\state)=1$ for every $\state\in \sset$. We will then write $\sset[A]\sacs[n]\sset[B]$. If there is some $n > 0$ such that $\sset[A]\sacs[n] \sset[B]$ then we will say that $\sset[A]$ \emph{strongly leads} to $\sset[B]$ and write $\sset[A]\sacs\sset[B]$.  
\end{definition}
The strong accessibility relation denotes the case where the chain being in any state belonging to $A$ will move to a set of states $B$ with upper probability 1. It should be noted that this does not mean that the chain will move from $A$ to $B$ with certainty, but only that one of possible scenarios, described with the transition probabilities, is also such a certain move. And here is the crucial difference between the precise and imprecise case. While in the precise case a set cannot strongly lead into two disjoint sets, this is perfectly possible in the imprecise case. This fact implies that the interdependencies between classes of states in the imprecise case can be more complex than in the precise case. 

Note also that the case where $A\sacs B$ is different from the case where $\{ \state\}\sacs B$ for every $\state\in A$, although implies it. Consider the following example. 
\begin{example}
 Let $\{ \statex \}\sacs[1] \{ \statey \}$ and $\{ \statey \}\sacs[1] \{ \statex \}$. Then $\{ \statex \}\sacs[2k] \{ \statex \}$ for every $k\in\NN$ and $\{ \statey \}\sacs[2k+1] \{ \statex \}$ for every $k\in\NN$. However, there is no $n\in \NN$ such that $\{ \statex, \statey\}\sacs[n] \{ \statex \}$, whence $\{ \statex, \statey\}\not\sacs \{ \statex \}$.
\end{example}
To help us analyse the properties of the strong accessibility relation $\sacs$ we will define the following two functions. Let $\ifnal$ be an IEF. Define 
\begin{equation*}
 \psets(\sset) = \begin{cases}
                     1 & \uifnal(\chf{\sset}) = 1; \\
                     0 & \text{otherwise}
                    \end{cases}       
\end{equation*}
for every $\sset \subseteq \states$. 

Further let $\toper$ be an ITO and define
\begin{equation*}
 \acsets(\sset[A], \sset[B]) = \begin{cases}
                                1 & \sset[A]\sacs[1]\sset[B]; \\
                                0 & \text{otherwise.}
                               \end{cases}
\end{equation*}
The following proposition follows immediately from the definitions.
\begin{proposition}\label{prop-psets-acsets}
Let $\ifnal$ be an IEF and $\sset[A], \sset[B], \sset[C], \sset[D], \sset[A]_i$ for $i=1, \ldots, n$, sets of states. 
 \begin{enumerate}[(i)]
 \item $\acsets(A, B) = 1$ if and only if $\psets[{\toper(\cdot|\state)}](\sset[B])=1$ for every $\state\in\sset$.
  \item If $\sset[A]\subseteq \sset[B]$ then $\psets(\sset[A])\le \psets({\sset[B]})$.
  \item If $\sset[C]\subseteq \sset[A]$ and $\sset[B]\subseteq \sset[D]$ then $\acsets[\toper](\sset[A], \sset[B]) \le \acsets[\toper](\sset[C], \sset[D])$.
  \item $\displaystyle\min_{i=1, \ldots, n}\acsets[\toper](\sset[A_i], \sset[B]) =  \acsets[\toper]\left(\bigcup_{i=1}^n\sset[A]_i, \sset[B]\right)$.
 \end{enumerate}
\end{proposition}
We define the following operations:
\begin{equation*}
 (\psets*\acsets)(\sset[B]) = \max_{\sset[A]\subseteq \states} \psets(\sset[A])\acsets(\sset[A], \sset[B])
\end{equation*}
and 
\begin{equation*}
 (\acsets*\acsets[\toper'])(\sset[A], \sset[B]) = \max_{\sset[C]\subseteq\states} \acsets(\sset[A], \sset[C])\acsets[\toper'](\sset[C], \sset[B]). 
\end{equation*}
The latter expression, when sequentially applied to the same function, allows defining its $n$-th power with
\begin{equation*}
 \acsets^n = \underbrace{\acsets *\cdots *\acsets}_{n-\text{times}}.
\end{equation*}
\begin{proposition}\label{psets-acsets-multi}
Let $\ifnal$ be an IEF, $\toper$ and $\toper[S]$ ITOs and $n$ a positive integer. Then
 \begin{enumerate}[{(i)}]
  \item $\psets[\ifnal\toper] = \psets[\ifnal]*\acsets[\toper];$
  \item $\acsets[{\toper \toper[S]}] = \acsets[\toper]*\acsets[{\toper[S]}];$
  \item $\acsets[{\toper^n}] = \acsets^n$;
  \item $\psets[{\ifnal\toper^n}] = \psets * \acsets^n$;
  \item $\displaystyle\psets[{[\lim_{n\to\infty}\ifnal\toper^n]}] \ge \lim_{n\to\infty}\psets * \acsets^n$ if the limits exist. Particularly, if $ET\subseteq E$, then we have the equality $\displaystyle\psets[{[\lim_{n\to\infty}\ifnal\toper^n]}] = \lim_{n\to\infty}\psets * \acsets^n$.
 \end{enumerate}
\end{proposition}
\begin{proof}
(i): Let us first show that $\psets[\ifnal]*\acsets[\toper]\le \psets[\ifnal\toper]$. Let $(\psets[\ifnal]*\acsets[\toper])(B) = 1$ for some $\sset[B]$. By definition, this implies existence of some $\sset[A]$ such that $\psets(\sset[A]) = 1$ and $\acsets(A, B) = 1$. In other words $\uifnal(\chf{A}) = 1$ and $\utoper \chf{B} (\state) = 1$ for every $\state\in \sset[A]$. Hence, by Lemma~\ref{x-krat-char} (i), $\uifnal\,\utoper(\chf{B}) = \uifnal(\utoper\chf{B}) \ge 1\cdot \uifnal(\chf{A}) = 1$, which means, by definition, that $\psets[\ifnal\toper](\sset[B])= 1$.
 
 Now suppose that $\psets[\ifnal\toper](\sset[B])= 1$ for some $\sset[B]$. By definition then $\uifnal\,\utoper(\chf{B}) = \uifnal(\utoper\chf{B}) = 1$, which by Lemma~\ref{x-krat-char}\,(ii) is only possible if $\uifnal(\chf{\{ \utoper\chf{B}=1\}}) = 1$. Then we denote $\sset=\{ \state \colon \utoper\chf{B}(\state)=1 \}$, which then satisfies $\psets(\sset) = 1$ and $\acsets(\sset, \sset[B]) = 1$, whence $(\psets * \acsets)(\sset[B]) = 1$, and this proves that $\psets * \acsets \ge\psets[\ifnal\toper]$. 
 
(ii): By Proposition~\ref{prop-psets-acsets}\,(i), $\acsets[{\toper \toper[S]}](A, B) = 1$ is equivalent to $\psets[{\toper \toper[S](\cdot|\state)}](B) = \psets[{\toper(\cdot|\state) \toper[S]}](B)= 1$ for every $\state\in A$. By (i), this is equivalent to $(\psets[{\toper(\cdot|\state)}]*\acsets[{\toper[S]}])(B) = 1$ for every $\state \in A$. Therefore, for every $\state \in A$ we have a set $C_{\state}$ such that $\psets[{\toper(\cdot|\state)}](C_{\state}) = 1$ and $\acsets[{\toper[S]}](C_{\state}, B) = 1$. Now let $C = \bigcup_{\state\in A} C_{\state}$ and, by Proposition~\ref{prop-psets-acsets}\,(ii) and (iv), we still have that $\psets[{\toper(\cdot|\state)}](C) = 1$ for every $\state\in  A$ and $\acsets[{\toper[S]}](C, B) = 1$. Hence, $\acsets(A, C) = 1$ and $\acsets[{\toper[S]}](C, B) = 1$, which implies $(\acsets[{\toper}]*\acsets[{\toper[S]}])(A, B) = 1$. 

Conversely, if we assume that $(\acsets[{\toper}]*\acsets[{\toper[S]}])(A, B) = 1$, this implies the existence of some $C$ such that $\acsets[{\toper}](A, C) = 1$ and $\acsets[{\toper[S]}](C, B) = 1$. By Proposition~\ref{prop-psets-acsets}\,(i), the first equality is equivalent to $\psets[{\toper(\cdot|\state)}](C) = 1$ for every $\state \in A$. Therefore $\psets[{\toper(\cdot|\state) \toper[S]}](B)= 1$ for every $\state\in A$; which, as follows from above, is equivalent to $\acsets[{\toper \toper[S]}](A, B) = 1$.

(iii) is an immediate consequence of (ii), and (iv) follows from (i) and (iii).

(v) Let $\lim_{n\to\infty}(\psets * \acsets^n)(\sset) = 1$. Then there must exist some $n_0$ such that $(\psets * \acsets^n)(\sset) = 1$ for every $n\ge n_0$ and this is equivalent to $\psets[{\ifnal\toper^n}](\sset) = 1$. Thus $\uifnal\,\utoper^n(A) = 1$ for every $n\ge n_0$, which implies that $\lim_{n\to\infty}\uifnal\,\utoper^n(\sset) = 1$ or $\displaystyle\psets[{[\lim_{n\to\infty}\ifnal\toper^n]}](\sset) = 1$.

If $ET\subseteq E$ then $ET^n\subseteq ET^{n-1}$ holds for every $n$. Now suppose that $\lim_{n\to\infty}\uifnal\,\utoper^n(\sset) = 1$. Then the same must hold for every $n\colon \uifnal\,\utoper^n(\sset) = 1$. Hence $\psets * \acsets^n(\sset)=1$ for every $n$ and so is the limit $\lim_{n\to\infty}\psets * \acsets^n(\sset)=1$.
\qed
\end{proof}
Let us give an example where we have strict inequality in Proposition~\ref{psets-acsets-multi}~(v). 
\begin{example}
 Take the transition matrix 
 \[ \toper = \begin{bmatrix}
         1 & 0 \\
         0.5 & 0.5
        \end{bmatrix}, \]
for a Markov chain with the set of states $\states = \{ \state_1, \state_2\}$. 
Clearly, starting with the initial probability distribution, say $q_0 = (0.5,  0.5)$, we have $\psets(A)=0$ for all strict subsets $A\subsetneq \states$.  The same holds for every distribution $q_n = q_0T^n$, but the limit distribution is $(1, 0)$. Hence, we have $\psets[{q_n}](A) = 0$ for every strict subset of $\states$, and the same for the limit $\lim_{n\to\infty} \psets[{q_n}](A)$. However, $\psets[{[\lim_{n\to\infty}q_n]}](\{ \state_1 \}) = 1$. 
\end{example}
\begin{corollary}\label{cor-sequence}
Let $A$ and $B$ be sets of states such that $A\sacs[n] B$. Then a sequence of sets 
\begin{equation}\label{eq-sequence-1}
A=A_0, A_1, \ldots, A_n = B 
\end{equation}
exists such that
\begin{equation}\label{eq-sequence-2}
A_i \sacs[1] A_{i+1}, \text{for every } i=0, \ldots, n-1.
\end{equation}
\end{corollary}
\begin{proof}
	The relation $A\sacs[n] B$ is equivalent to $\acsets[\toper^n](A, B) = 1$, and by Proposition~\ref{psets-acsets-multi}\,(iii), this is again equivalent to $\acsets^n(A, B) = 1$. The existence of the sequence is now an immediate consequence of the definitions. 
\qed
\end{proof}
\subsection{Permanent classes}\label{ss-pc}
\begin{definition}
A set $\sset$ of states is called a \emph{permanent class} if for every $N\in\NN$ an $n\ge N$ and a set $B$ exist such that $\sset[B]\sacs[n] \sset$. Any set of states that is not permanent will be called an \emph{impermanent class}.
\end{definition}
In the case of precise Markov chains only absorbing classes are permanent, while the imprecise case allows much richer structure. 
\begin{example}
 Consider the transition operator from Example~\ref{ex-1}. Clearly, the disjoint subsets of states $\states_1$ and $\states_2$ are both permanent classes, despite the fact that $\states_2$ 'possibly' leads to $\states_1$. 
\end{example}
The following proposition is immediate. 
\begin{proposition}
 Every subset of states $\sset$ such that $\sset\sacs\sset$ is a permanent class. 
\end{proposition}
\begin{definition}
A permanent class $\sset[B]$ is \emph{minimal} if $\sset[B]\sacs\sset[B]$ and contains no proper subset $\sset[B]'\subset \sset[B]$ such that $\sset[B]'\sacs\sset[B]'$.
\end{definition}
We now examine some properties of permanent classes. First we give some equivalent definitions. 
\begin{proposition}\label{prop-pc-equivalent}
Let a set of states $A$ be given. The following propositions are equivalent:
\begin{enumerate}[{(i)}]
\item $A$ is a permanent class. 
\item For every $N\in \NN$ an $n\ge N$ exists and a sequence of sets $A_0, A_1, \ldots, A_n = A$ such that $A_i \sacs[1] A_{i+1}$ for every $i=0. \ldots, n-1$. 
\item For every $n\in \NN$ a sequence of sets with the above properties exists.
\item A permanent class $B$ exists such that $B\sacs A$.
\item A minimal permanent class $B$ exists such that $B\sacs A$.
\end{enumerate}
\end{proposition}
\begin{proof}
	The equivalence between (i), (ii) and (iii) clearly follows by Corollary~\ref{cor-sequence}.
	Further, (iv) and (v) clearly imply (iii). 	
	
	Now let $A$ be a permanent class. Then for every $n$ we have a sequence $A_{n,0}, \ldots, A_{n,n-1}$ such  that $A_{n,i}\sacs[1] A_{n, i+1}$ for every $n$ and $i<n$ and $A_{n, n-1}\sacs[1] A$. Because of finiteness, there exists a set $B$ such that $B=A_{n, n-1}$ for an infinite number of indices $n$, and this implies that $B$ satisfies (ii) and is therefore a permanent class. Thus, (i) implies (iv) as well. 
	
	To see that (iv) implies (v), let $A$ be a permanent class and $\sset[B]_1\sacs \sset[B]_2 \sacs \cdots \sacs\sset[B]_r = \sset$ a maximal chain such that $\sset[B]_{k+1}\not\sacs \sset[B]_k$ for every $k=1, \ldots, r-1$. By (iv) there exists a permanent class $\sset[C]$ such that $\sset[C]\sacs \sset[B]_1$. But then, by maximality of the chain, we must have that also $\sset[B]_1\sacs \sset[C]$, and consequently, by transitivity, $\sset[B]_1\sacs \sset[B]_1$. If there is a strict subset $\sset[B]'\subset \sset[B]_1$ that is a permanent class then we take it to be a minimal one. We have then $\sset[B]'\sacs\sset[B]_1$, by construction. But $\sset[B]_1\sacs \sset[B]'$, by maximality of the chain. Therefore, by transitivity, we again have that $\sset[B]'\sacs \sset[B]'$ which is now clearly a minimal permanent class, and such that $\sset[B]'\sacs \sset$. 
\qed
\end{proof}
\begin{proposition}\label{prop-mpc-com}
 Let $\sset[B]$ be a minimal permanent class. Then it is a communication class. Moreover, there is some $r\in \NN$ such that 
 \begin{equation}\label{eq-power-regular}
 	\utoper^r\chf{\{ \statey \}} (\statex) >0 \quad \text{for every }\statex, \statey \in \states
 \end{equation}
 and 
 \begin{equation}\label{eq-mpc-com-2}
 \utoper^r \chf{B}(\state) = 1 \quad \text{for every } \state\in B.
 \end{equation} 
 This means that for some $r\in \NN$ the chain with the ITO $\toper^r$ restricted to $B$ is regular.
\end{proposition}
\begin{proof}
Let us first show that $B$ is a communication class.  We have that $\sset[B]\sacs\sset[B]$, say $\sset[B]\sacs[n]\sset[B]$ for some $n\in \NN$.
 Take some $\state\in\sset[B]$ and denote $\tilde{\sset[B]}_{\state} = \{ \statey\in\sset[B]\colon \state\not\acs \statey \}$ and further let $\sset[B]_{\state} = \{ \state'\in \sset[B] \colon \state' \not\acs \tilde{\sset[B]}_{\state}\}$. We show that $\sset[B]_{\state}\sacs[n] \sset[B]_{\state}$. Suppose contrary that there is some $\state'\in \sset[B]_{\state}$ such that $\utoper^n\chf{\sset[B]_x}(\state')<1$. By subadditivity of $\utoper^n(\cdot|\state')$ we have that $\utoper^n\chf{\sset[B]\backslash\sset[B]_x}(\state')\ge \utoper^n\chf{\sset[B]}(\state')-\utoper^n\chf{\sset[B]_x}(\state')>0$. Hence, there must be some $\statey\in \sset[B]\backslash \sset[B]_x$ such that $\utoper^n\chf{\{ \statey \}}(\state')>0$, or equivalently, $\state'\acs[n]\statey$. But since $\statey\not\in\sset[B]_x$, $\statey\acs\tilde{\sset[B]}_{\state}$, which together implies that $\state'\acs \tilde{\sset[B]}_{\state}$, which is in contradiction with $\state'\in \sset[B]_x$. 
 
 It is therefore confirmed that $\sset[B]_{\state}\sacs \sset[B]_{\state}$, which contradicts minimality of $\sset[B]$, unless $\sset[B]_{\state}=\emptyset$. Together this proves that $B$ is a communication class. 
 
Let us now show that some $r\in \NN$ exists so that equations \eqref{eq-power-regular} and \eqref{eq-mpc-com-2} are satisfied. Let $\state\in B$ and denote with $r_{\state}$ the greatest common divisor of all $r$ such that $\state\acs[r]\state$. Now denote $C_{\state} = \{\statey\in B\colon  \exists k, \state \acs[kr_{\state}] \statey \}$. 

Since $\utoper^{kr_{\state}}\chf{B\backslash C_{\state}} (\statey) = 0$ for every $\statey\in C_{\state}$, which implies that $\utoper^{kr_{\state}}\chf{B_{\state}} (\statey) = 1$ for every $\statey\in C_{\state}$. Thus, $C_{\state}\sacs[kr_{\state}]C_{\state}$ and therefore, $C_{\state}\sacs C_{\state}$. Since $B$ is a minimal permanent class and $C_{\state}$ a non-empty subset, it must therefore hold that $C_{\state} = B$ for every $\state\in B$. 
 
We now show that there is some $r\in\NN$ such that $\utoper^r\chf{\{ \statey \}}(\state)>0$ for every $\state, \statey\in B$. It follows from the definition of $r_{\state}$ and the elementary properties of accessibility relations that there is some $L\in \NN$ such that $\state\acs[lr_{\state}]\state$ for every $l\ge L$, and for every $\statey\in B$ there exists some $l_{\statey}$ such that $\state\acs[l_{\statey}r_{\state}]\statey$. Denote $M_{\state} = \max_{\statey\in B}(L + l_{\statey})$. Then $\state\acs[mr_{\state}]\statey$ for every $\statey\in B$ and every $m\ge M_{\state}$. Now denote $R=\prod_{\state\in B} r_{\state}$ and $M=\max_{\state\in B} M_{\state}$. Clearly then $\statex \acs[mR]\statey$ for every $\statex$ and $\statey\in B$, whenever $m\ge M$. 

Because  $B$ is a minimal permanent class, we have that $B\sacs[s]B$. If $s=1$ then \eqref{eq-mpc-com-2} is satisfied for every $r\in \NN$ and therefore also for $r=mR$. While in the case where $s>1$ we can replace $\toper$ with $\toper^s$ and reason as above to show that $(\toper^s)^{mR}=\toper^{smR}$ satisfies both properties required, since, as can easily be noted, $B$ is still a minimal permanent class when $\toper$ is replaced with $\toper^s$.
\qed
\end{proof}
Note that the above proposition does not claim that a minimal permanent classes would be whole communication classes, but merely that every minimal permanent class lies in a single communication class. 
\begin{definition}
Let $B$ be a minimal permanent class, such that $B\sacs[r] B$. Then we define 
\begin{equation}\label{eq-TBr}
\toper^r_B := \{ \ptoper\in \toper^r\colon \ptoper \chf{B}(\state) = 1 \text{ for every }\state\in B \}.
\end{equation}
\end{definition}
\begin{proposition}\label{prop-mpc-abs-reg}
	Let $B$ be a minimal permanent class, such that $B\sacs[r] B$. Then $B$ is regular and absorbing with respect to the weak accessibility relation induced by $(\toper_B^r)^k$ for some $k\in \NN$. 
\end{proposition}
\begin{proof}
Let $\acs$ and $\sacs$ now denote the weak and the strong accessibility relation induced by $\toper_B^r$. By construction we have that $B\sacs B$ and since $\toper_B^r$ is a subset of $\toper^r$ this relation cannot hold for any proper subset of $B$. Therefore $B$ is a minimal permanent class with respect to $\toper_B^r$ as well. 

By Proposition~\ref{prop-mpc-com} it is then a communication class and, moreover, for some $k\in \NN$ the power $(\toper_B^r)^k$ restricted to $B$ is regular. By construction, $B$ is also absorbing with respect to $\toper_B^r$ and therefore, with respect to $(\toper_B^r)^k$ as well. 
\qed
\end{proof}
\section{Unique convergence for extremal IEFs}\label{s-cuc}
In this section we use the results prepared in previous sections to state our main results on convergence of imprecise Markov chains. We will explore the convergence of the class of so called extremal IEFs, which are those which induce the upper probabilities of minimal permanent classes being either $0$ or $1$. In the next subsection we show that when the sequence $\{ \ifnal \toper^n \}$ consists of extremal IEFs this sequence converges if the sequence $\ifnal(\toper^n \chf{B})$ converges (either to $0$ or $1$) for every minimal permanent class $B$, and the limit is uniquely determined by the limits $\lim_{n\to\infty}\ifnal(\toper^n\chf{B})$ for all minimal permanent classes $B$. 

Moreover, in Subsection~\ref{ss-impc} we show that it is sufficient for the unique convergence that the limits $\lim_{n\to\infty}\ifnal(\toper^n\chf{B})$ are zero-one valued, rather than all terms of the sequences. 
\subsection{Uniqueness of invariant imprecise expectation functionals}\label{ss-uiief}
The importance of permanent classes is illustrated with the following proposition. 
\begin{proposition}
  Let $\ifnal[M]$ be the unique least committal $\toper$-invariant IEF. Then $\psets[{\ifnal[M]}](\sset) = 1$ if and only if $\sset$ is a permanent class.
\end{proposition}
\begin{proof}
The unique least committal IEF $\ifnal[M]$ is equal to $\lim_{n\to\infty}\ifnal[V]\toper^n$, where $\ifnal[V]$ is the vacuous IEF on $\states$. Since $\ifnal[V]\toper \subseteq \ifnal[V]$, Proposition~\ref{psets-acsets-multi}\,(v) gives that 
\[ \psets[M] = \psets[{[\lim_{n\to\infty}\ifnal[V]\toper^n]}] = \lim_{n\to\infty} \psets[V]*\acsets^n. \]
If $\psets[M](A) = 1$, this implies that for every $n\in\NN$ we have that $\psets[V]*\acsets^n(A)=1$, or equivalently, there exists some $B$ such that $\psets[V](B)=1$ and $B\sacs[n] A.$ Hence, $A$ is a permanent class. 

Moreover, if $A$ is a permanent class, then for every large enough $n\in\NN$ there exists some $B$ such that $B\sacs[n] A$, and since $\psets[V](B)=1$ for every $B\subseteq \states$, then $\psets[V]*\acsets^n(A)=1$. As this holds for every large enough integer $n$, so it does for the limit, and implies that $\psets[M](A)=1$. 
\qed
\end{proof}
\begin{definition}\label{def-essmax-me}
	Let $\ifnal$ be an IEF and $\gamble$ a gamble. Then we define \emph{essential maximum} of $\gamble$ given $\ifnal$ with
	\[ \essmax_{\ifnal}\gamble = \max\{ a\colon \lifnal(\chf{\{f\ge a\}})>0 \}, \]
and
	\[ m_{\ifnal} = \min\{ \lifnal(\chf{A})\colon A\subseteq \states, \lifnal(\chf{A})>0 \}. \]
\end{definition}
\begin{proposition}\label{p_es_m}
Let $\ifnal$ and $\ifnal[F]$ be IEFs and $\gamble$ a gamble. The following propositions hold:
\begin{enumerate}[{(i)}]
\item $\lifnal(\gamble) \le \essmax_{\ifnal}f$.
\item If $f\ge 0$ then $\lifnal(\gamble) \ge m_{\ifnal}\cdot \essmax_{\ifnal}f$.
\item If $\psets = \psets[F]$, then $\essmax_{\ifnal[F]}f= \essmax_{\ifnal}f$. 
\item Let $\{ \ifnal_n \}$ be a sequence of IEFs such that $\ifnal_{n+1}\subseteq \ifnal_n$. Then $\lim_{n\to \infty} m_{\ifnal_n}$ exists and there is an $N\in\NN$ such that $\{ m_{\ifnal_n} \}_{n\ge N}$ is a non-decreasing sequence.
\end{enumerate}
\end{proposition}
\begin{proof}
Let in this proof $\tilde\gamble$ denote $\essmax_{\ifnal} \gamble$.

To see (i), take some $\fnal\in \ifnal$ such that $\fnal(\chf{\{\gamble>\tilde \gamble\}})=0$. This means that $\fnal(\chf{\{\state\}}) = 0$ for every $\state\in\states$ such that $\gamble(\state)>\tilde \gamble$. Then 
  \[ \lifnal(\gamble) \le \fnal(\gamble) = \sum_{\state\in\states} \fnal(\state) \gamble(\state) = \sum_{\state\colon \gamble(\state)\le \tilde\gamble} \fnal(\state) \gamble(\state)\le \tilde\gamble . \]
  
By Lemma~\ref{x-krat-char}~(i) and (i) of this proposition, we obtain
  \begin{equation*}
    \lifnal(\gamble) \ge \tilde\gamble\lifnal(\chf{\{\gamble\ge \tilde \gamble\}}) \ge \tilde \gamble m_{\ifnal},
  \end{equation*}
which proves (ii).

To see (iii) notice that $\psets = \psets[F]$ if and only if $\lifnal(\chf{\{\gamble>a\}}) > 0$ whenever $\lifnal[F](\chf{\{\gamble>a\}}) > 0$.

(iv): To see this note that the set $\{ \lifnal_n(\chf{\sset})\colon n\in \NN\}$ is a non-decreasing sequence for every $\sset\subseteq\states$, either constantly equal to 0, or contains a minimal non-zero element. Because of the finite number of subsets of $\states$ we have at most a finite number of positive minima, whose minimum exists and is positive as well. 
	
	Clearly, there is also an $N\in \NN$ such that  for every $\sset\subseteq \states$ either $\lifnal_n(\chf{\sset})>0$ for every $n\ge N$ or $\lifnal_n(\chf{\sset})=0$ for every $n\ge N$. Therefore, $\{ m_{\ifnal_n}\}_{n\ge N}$ is a non-decreasing sequence.  
\qed
\end{proof}
\begin{corollary}\label{cor-max-quotient}
	Let $\ifnal$ and $\ifnal[F]$ be IEFs such that $\psets = \psets[F]$. Then 
	\begin{equation}\label{es_m_quotient}
		 \lifnal(\gamble) \ge  m_{\ifnal}\cdot \lifnal[F](\gamble)
	\end{equation}	
	for every gamble $\gamble\ge 0$.
\end{corollary}
\begin{proof}
	By Proposition~\ref{p_es_m}(ii) we have that $\lifnal(\gamble) \ge m_{\ifnal}\cdot \essmax_{\ifnal}f$, and by (i) of the same proposition we have that $\lifnal[F](\gamble) \le \essmax_{\ifnal[F]}f = \essmax_{\ifnal}f$. This together implies \eqref{es_m_quotient}. 	
\qed
\end{proof}
\begin{corollary}\label{cor-vacuous-mixture}
	Let $\ifnal$ and $\ifnal[F]$ be IEFs such that $\psets = \psets[F]$, $\supp(\ifnal[F]) \subseteq \supp(\ifnal[E])$, and $\ifnal[V]_{\ifnal}$ the vacuous IEF on $\supp(\ifnal)$. Then 
\[ 	\ifnal \subseteq m_{\ifnal}\ifnal[F] + (1-m_{\ifnal})\ifnal[V]_{\ifnal}. \]
\end{corollary}
\begin{proof}
We need to prove that 	
\begin{equation}\label{ineq-vacuous-mixture}
		\lifnal (\gamble) \ge m_{\ifnal} \lifnal[F](\gamble) + (1-m_{\ifnal}) \lifnal[V]_{\ifnal}(\gamble)
\end{equation}
holds for every gamble $\gamble\in \allgambles(\states)$. Let $\gamble$ be a gamble and denote $h=\gamble\cdot \chf{\supp(\ifnal)}$. Then, by Proposition~\ref{gamble-support}, we have that $\ifnal[H](\gamble) = \ifnal[H](h)$ where $\ifnal[H]$ stands for $\ifnal, \ifnal[F]$ and $\ifnal[V]_{\ifnal}$ respectively.  

Denote $\underline{h}=\min_{\state\in\supp(\ifnal)}\gamble(\state) = \lifnal[V]_{\ifnal}(\gamble)$. Then we have that $h-\underline{h}\ge 0$. By Corollary~\ref{cor-max-quotient} we have that $\lifnal(h-\underline{h}) \ge m_{\ifnal}\lifnal[F](h-\underline{h})$. Hence, by constant additivity of $\lifnal$ and $\lifnal[F]$, $\lifnal(h)-\underline{h} \ge m_{\ifnal}\lifnal[F](h)-m_{\ifnal}\underline{h}$ holds. By replacing $\underline{h}$ with $\lifnal[V]_{\ifnal}(\gamble), \lifnal(h)$ with $\lifnal(\gamble)$ and $\lifnal[F](h)$ with $\lifnal[F](\gamble)$ we obtain 
\[ \lifnal(f)-\lifnal[V]_{\ifnal}(\gamble) \ge m_{\ifnal}\lifnal[F](\gamble)-m_{\ifnal}\lifnal[V]_{\ifnal}(\gamble), \]
whence inequality \eqref{ineq-vacuous-mixture} follows. 
\qed
\end{proof}
\begin{definition}\label{def-extremal}
	Let $\ifnal$ be an IEF such that $\uifnal(\chf{B})\in \{0, 1\}$ for every minimal permanent class $B$. Then we say that $\ifnal$ is an \emph{extremal} IEF. 
\end{definition}
In the precise case the concept of an extremal expectation functional is rather trivial. Note that a minimal permanent class in the precise case can only be absorbing. Therefore having the entire probability mass once concentrated in such a class it will remain there forever, and additionally, it can only be concentrated in one such class at the time.  Consequently, if the class is aperiodic, the probability distributions will then converge to a unique distribution invariant for the chain restricted to this, regular and absorbing class. 
\begin{definition}\label{def-SE}
	Let $\ifnal$ be an IEF. Then we define 
	\[ S_{\ifnal} = \{ \statey \colon \exists \statex \in \supp(E), \statex\acs\statey \} \]
	and $V_{\ifnal}$ will from now on denote the vacuous IEF on $S_{\ifnal}$. 
\end{definition}
Clearly, the following holds. 

\begin{proposition}\label{prop-least-committal}
	For every IEF $\ifnal$, the set $S_{\ifnal}$ is absorbing. Moreover, the sequence 
	\begin{equation*}\label{en-sequence-Mn}
		M_n = V_{\ifnal}\toper^n
	\end{equation*}
	is monotone: $M_{n+1}\subseteq M_n$, and the limit
	\begin{equation*}\label{en-M}
		M = \lim_{n\to\infty} M_n	
	\end{equation*}
	exists and is $\toper$-invariant. 
\end{proposition}
\begin{proof}
It is an immediate consequence of the definition that $S_{\ifnal}$ is absorbing. The convergence then follows from Propositions~\ref{limit-invar} and \ref{abs-monotone}.
\qed
\end{proof}
\begin{lemma}\label{lem-permanent-absorbing}
	Let $\sset$ be a permanent class such that $\sset \sacs[r]\sset$ and $S$ an absorbing set. Then $\sset\cap S\sacs[r]\sset\cap S$. 
\end{lemma}
\begin{proof}
	Clearly, $S$ being absorbing implies that $\utoper^r \chf{C}(\state) = 0$ for $C\cap S= \emptyset$ and $\state\in S$. Now take some $\state\in \sset\cap S$. By subadditivity of $\utoper(\cdot|\state)$ we have that $1=\utoper^r \chf{\sset}(\state) \le \utoper^r \chf{\sset\cap S}(\state)+ \utoper^r \chf{\sset\backslash S}(\state)$; whence, by $\utoper^r \chf{\sset\backslash S}(\state)=0$, $\utoper^r \chf{\sset\cap S}(\state)=1$ follows. 
\qed
\end{proof}
\begin{corollary}\label{cor-mpc-contained}
	Let $B$ be a minimal permanent class and $S$ an absorbing set. Then $B\cap S\in \{B, \emptyset \}$. 
\end{corollary}
\begin{proposition}\label{prop-monotonicity-psi}
Let $A\sacs[n] B$ and let $\ifnal$ be an IEF. Then $\uifnal(\utoper^n\chf{B})\ge \uifnal(\chf{A})$.
\end{proposition}
\begin{proof}
We have that $\uifnal(\utoper^n\chf{B})\ge \uifnal(\chf{A})\cdot\min_{\state\in A}\utoper^n\chf{B}(\state) = \uifnal(\chf{A})$.
\qed
\end{proof}
\begin{corollary}
Let $A\sacs B$ and let $\ifnal$ be a $\toper$-invariant IEF. Then $\uifnal(\chf{B})\ge \uifnal(\chf{A})$.
\end{corollary}
\begin{lemma}\label{lema-unique-psi}
	Let $\ifnal$ be an IEF such that $\ifnal \toper^n$ is extremal for every $n\ge 0$ and the limit $\lim_{n\to\infty}\psets[\ifnal\toper^n](B)$ exists for every minimal permanent class $B$. Then the limit
	\[ \psets[](A) := \lim_{n\to\infty}\psets[\ifnal\toper^n](A) \]
	exists for every $A\subseteq \states$ and $\psets[]$ is uniquely determined by its restriction to minimal permanent classes. 
\end{lemma}
\begin{proof}
	Denote $\ifnal_n = \ifnal\toper^n$. Let us fist show that for every minimal permanent class $B$ $\lim_{n\to\infty}\psets[\ifnal_n](B)=1$ if and only if $B\subseteq S_{\ifnal}$. To see this, take some minimal permanent class $B\subseteq S_{\ifnal}$ and note that $\uifnal_k(\chf{B})>0$ must hold by definition of $S_{\ifnal}$ for some $k\ge 0$. Extremality of $\ifnal_k$ implies that $\uifnal_k(\chf{B})=1$, and the fact that $B\sacs B$, say $B\sacs[r] B$, by Proposition~\ref{prop-monotonicity-psi}, implies that $\uifnal_{k+mr}(\chf{B})=1$, for every $m\in \NN$, whence, by the assumed convergence of the sequence $\{ \psets[\ifnal_n](B)\}$, we must have that, for some $N\in \NN,$ $\psets[\ifnal_n](B)=1$ for every $n\ge N$. 
	
On the other hand, if $B\not\subseteq S_{\ifnal}$, then by Corollary~\ref{cor-mpc-contained}, $B\cap S_{\ifnal}=\emptyset$, and therefore clearly, $\uifnal_n(\chf{B})=0$ for every $n\ge 0$. 	
	
	 Now let $A\subseteq \states$. Then either $\psets[{\ifnal_n}](A) = 0$ for all sufficiently large $n$ or for every $N\in\NN$ there exists an $n\ge N$ such that $\psets[{\ifnal_n}](A) = 1$. The former case implies convergence, so let us show that the latter case does as well. Suppose that $\psets[{\ifnal_n}](A)=1$. By Proposition~\ref{psets-acsets-multi}\,(iv) we have that $\psets[{\ifnal_n}](A) = (\psets * \acsets^n)(A)=1$, which implies the existence of a sequence $A_0\sacs  A_1\sacs \ldots \sacs A_n = A$, such that $\psets(A_0) = 1$. If $n$ is sufficiently large, then at least two members of the sequence are equal, say $A_k = A_l = A'$. Hence, we have $A'\sacs A'\sacs A$. By Lemma~\ref{lem-permanent-absorbing} we have that $B' = A'\cap S_{\ifnal}$ also satisfies $B'\sacs B'$, and moreover $B'\sacs A'\sacs A$. However, $B'$ either is a minimal permanent class or contains one. Let $B$ denote such a minimal permanent class, which by construction lies within $S_{\ifnal}$, and therefore, by the assumptions, an $N\in\NN$ exists such that  $\psets[\ifnal_n](B)=1$ for every $n\ge N$. 	
	We have that $B\sacs A$, say $B\sacs[r] A$. Moreover, by Proposition~\ref{prop-monotonicity-psi}  we have that $\psets[{\ifnal_{r+n}}](A) \ge \psets[{\ifnal_n}](B)=1$ for every $n\ge N$. Hence, $\psets[{\ifnal_{k}}](A)=1$ for every $k\ge r+N$.
\qed
\end{proof}

\begin{theorem}\label{thm-main-1}
	Let $\ifnal$ be an IEF such that $\ifnal_n = \ifnal \toper^n$ is extremal for every $n\ge 0$ and the limit $\lim_{n\to\infty}\psets[\ifnal_n](B)$ exists for every minimal permanent class $B$. Then the limit 
	\[ \ifnal_\infty := \lim_{n\to\infty}\ifnal_n \]
	exists. Moreover, $\lim_{n\to\infty}\ifnal_n =\lim_{n\to\infty}\ifnal[F]_n$ if and only if $S_{\ifnal}=S_{\ifnal[F]}$. 
\end{theorem}
\begin{proof}
	By Lemma~\ref{lema-unique-psi} and the assumptions of the theorem, the limit $\psets[] = \lim_{n\to\infty}\psets[\ifnal_n]$ exists, and since $\psets[\ifnal_n]$ are discrete valued, there is some $N_1\in \NN$ such that $\psets[\ifnal_n] = \psets[]$ for every $n\ge N_1$. 	
	
It has been shown in the proof of Lemma~\ref{lema-unique-psi} that $\psets[](B)=1$ for exactly those minimal permanent classes $B$ that are contained in $S_{\ifnal}$. 
	
	Let $V_{\ifnal}$ be the vacuous IEF on $S_{\ifnal}$ and denote $M_n = V_{\ifnal}\toper^n$ and $M=\lim_{n\to\infty}M_n$, which exists by Proposition~\ref{prop-least-committal}. Similar arguments as above show that for every minimal permanent class $B$ we have that $\psets[M](B) = 1$ if and only if $B\subseteq S_{\ifnal}$, whence $\psets[M] = \psets[]$. Moreover, by Proposition~\ref{psets-acsets-multi}\,(v), $\psets[M] = \lim_{n\to\infty}\psets[M_n]$, and therefore, there is some $N_2\in\NN$ such that $\psets[M_n] = \psets[M] = \psets[]$ for every $n\ge N_2$. 
	
	Because of $\ifnal \subseteq V_{\ifnal}$, we also have that $\supp(\ifnal_n) \subseteq \supp(M_n)$ for every $n\in \NN$. Now let $N=\max\{N_1, N_2 \}$ and denote $m = m_{M_N}$ (c.f. Definition~\ref{def-essmax-me}). By Corollary~\ref{cor-vacuous-mixture} we then have that 
	\begin{equation*}\label{eq-thm2-1}
		M_N \subseteq m \ifnal_N + (1-m)V_E.
	\end{equation*}
	Now take any gamble $\gamble\in \allgambles$ and any $n\ge 0$. By Proposition~\ref{convex-ief} we obtain
	\begin{equation*}\label{eq-thm2-2}
		\uifnal[M]_N(\utoper^n f) \le m \uifnal_N(\utoper^n f) + (1-m)\uifnal[V]_E(\utoper^n f).
	\end{equation*}
	Letting $n\to\infty$ in the above equation we obtain
	\begin{equation*}\label{eq-thm2-3}
		\uifnal[M](f) \le m \liminf_{n\to\infty}\uifnal_n(f) + (1-m)\uifnal[M](f), 
	\end{equation*}
	implying that $\uifnal[M](f) \le \liminf_{n\to\infty}\uifnal_n(f)$. But since, clearly, $\uifnal[M]_n(f)\ge \uifnal_n(f)$ for every $n\in \NN$, this is only possible if the limit 
	\[ \uifnal_\infty(f) = \lim_{n\to\infty}\uifnal_n(f) \]
	exists and is equal to $\uifnal[M](f)$. Since this holds for every gamble $\gamble$, we have that $\ifnal_\infty = \ifnal[M]$. Now since $\ifnal[M]$ is uniquely determined by $S_{\ifnal}$, the limit is the same for all IEFs $\ifnal$ that have the same $S_{\ifnal}$. 
\qed
\end{proof}

\subsection{Unique convergence to extremal invariant IEFs}\label{ss-impc}
In the previous subsection we have shown that extremal IEFs converge to a unique invariant IEF that is uniquely determined with the behaviour of the corresponding $\psets$ restricted to minimal permanent classes. But often the probability mass is not initially entirely concentrated in a minimal permanent class, but rather it accumulates there as time goes to infinity. It may then happen that in the limit the upper probabilities of all minimal permanent classes are either $0$ or $1$, although in any finite time they can be somewhere between. In this subsection we show that a convergence to an extremal IEF is unique whenever the upper expectations on minimal permanent classes converge to either $0$ or $1$. Moreover the limit IEFs are still uniquely determined with the behaviour on the minimal permanent class structure. 
\begin{proposition}\label{prop-absorbing-indicator}
	Let $\sset$ be an absorbing set. Then $\chf{\sset}\cdot \toper^n \gamble = \chf{\sset}\cdot \toper^n[\chf{\sset}\cdot\gamble]$ for every $n\in\NN$. 
\end{proposition}
\begin{proof}
	The fact that $\sset$ is absorbing implies that $\toper\chf{\sset^c}(\state) = 0$ for every $\state\in \sset$, and consequently, $\toper\gamble(\state) = \toper{[\chf{\sset}\gamble]}(\state)$ for every $\state\in\sset$. Hence
\begin{equation}\label{eq-gamble-times-indicator}
\chf{\sset}\toper\gamble = \chf{\sset}\toper[T]{[\chf{\sset}\gamble]}. 
\end{equation} 
But if $\sset$ is absorbing for $\toper$, then it is also absorbing for every $\toper^n$, where $n\in\NN$, and therefore the equation \eqref{eq-gamble-times-indicator} holds when $ \toper$ is replaced with $\toper^n$. 
\qed
\end{proof}
\begin{proposition}\label{prop-regular-convergence}
	Let $A$ be an absorbing set such that the weak accessibility relation induced by $\toper$ and restricted to $A$ is regular, and $\ifnal$ an IEF such that 
	\begin{equation}\label{eq-lower-one}
		\supp(\ifnal)\subseteq A.
	\end{equation}
	Then 
	\begin{equation*}
		\ifnal[M]_A := \lim_{n\to\infty} \ifnal \toper^n
	\end{equation*}
 exists and is the same for every $\ifnal$ satisfying \eqref{eq-lower-one}. 
\end{proposition}
\begin{proof}
	Let us first define the following transition operator:
	\[ \toper_A\gamble(\state) = \begin{cases}
		\toper \gamble(\state) & \state\in A; \\
		\ifnal[V]_A \gamble & \state \not\in A, 
	\end{cases}
	\]	
	where $\ifnal[V]_A$ is the vacuous IEF on $A$. Clearly, the chain with the ITO $\toper_A$ is regularly absorbing with the regular top class $A$. Therefore, by Theorem~\ref{conv-decooman}, the sequence $\{ \ifnal \toper_A^n \}$ converges uniquely to some $\ifnal[M]_A$. 
	Furthermore, Proposition~\ref{prop-absorbing-indicator} and the definition of $\toper_A$ imply that $\chf{A}\toper_A^n \gamble =\chf{A}\toper_A^n [\chf{A}\gamble] =\chf{A}\toper^n [\chf{A}\gamble]= \chf{A}\toper^n \gamble$ for every gamble $\gamble$. Using Proposition~\ref{gamble-support} and the above equalities we obtain
	\begin{align*} 
		\ifnal(\toper^n\gamble)  =  \ifnal (\chf{A}\toper^n\gamble) 
		 = \ifnal(\chf{A}\toper_A^nf) 
		 = \ifnal(\toper_A^nf),
	\end{align*}	
	and since the right hand sides converge, so do the left ones, for every $\gamble$. 
\qed
\end{proof}

\begin{corollary}\label{cor-minimalIEF}
	Let $A$ be a set of states and $\toper'\subseteq \toper$ an ITO such that $A$ is absorbing and regular with respect to the weak accessibility relation induced by $\toper'$. Then there exists a $\toper$-invariant IEF $\ifnal[M]_A$ such that for every IEF $\ifnal$ with $\uifnal(\chf{A})= 1$ we have that 
	\begin{equation}\label{eq-minimalIEF}
	\liminf_{n\to\infty}\uifnal(\utoper^n\gamble)\ge \uifnal[M]_A(\gamble)
	\end{equation}	 
	for every gamble $\gamble$. 
\end{corollary}
\begin{proof}
Since $\uifnal(\chf{A})=1$, there exists an IEF $\ifnal'\subseteq\ifnal$ such that $\supp (\ifnal')\subseteq A$. Denote $\ifnal'_n = \ifnal'\toper'^n$. By Proposition~\ref{prop-regular-convergence}, the limit $\ifnal[M]'_A := \lim_{n\to\infty} \ifnal'_n$ exists and is the unique such limit IEF corresponding to $\toper'$. Moreover $\ifnal[M]'_A = \ifnal[M]'_A\toper' \subseteq \ifnal[M]'_A\toper$ and therefore, the limit $\ifnal[M]_A = \lim_{n\to\infty} \ifnal[M]'_A \toper^n$ exists as well, by Proposition~\ref{limit-invar}, and is $\toper$-invariant. We will show that it satisfies \eqref{eq-minimalIEF}. 

The convergence of the sequence $\{ \ifnal'_n\}$ to $\ifnal[M]'_A$ implies that for every $\varepsilon > 0$ there exists an $N\in \NN$ such that $\uifnal'_n(\gamble + \varepsilon\chf{\states}) = \uifnal'_n(\gamble) + \varepsilon \ge \uifnal[M]'_A(\gamble)$ for every $n\ge N$ and every $\gamble \in \allgambles_1$, which is a compact set of gambles. Further, we denote $\ifnal_n = \ifnal \toper^n$, and clearly $\ifnal_n \supseteq \ifnal'_n$ holds. 
Hence, 
\begin{equation*}
\uifnal_{N+k}(\gamble) +  \varepsilon = \uifnal_{N+k}(\gamble +  \varepsilon\chf{\states}) \ge\uifnal'_{N+k}(\gamble +  \varepsilon\chf{\states}) = \uifnal'_{N+k}(\gamble) +  \varepsilon \ge \uifnal[M]'_A (\toper^k \gamble).
\end{equation*}
Letting $k\to\infty$ on both sides we obtain that 
\[ \liminf_{n\to\infty}\uifnal_n(\gamble) + \varepsilon\ge \uifnal[M]_A(\gamble) \]
for every $\varepsilon>0$, and therefore $\liminf_{n\to\infty}\uifnal_n(\gamble) \ge \uifnal[M]_A(\gamble)$ for every $\gamble\in \allgambles_1$, and therefore also for every gamble $f\in \allgambles$. 
\qed
\end{proof}

\begin{corollary}\label{cor-minimalIEF-limit}
	Let $A, \toper'$ and $\ifnal[M]_A$ satisfy conditions as in Corollary~\ref{cor-minimalIEF}, and let $\ifnal$ be an IEF such that 
	\begin{equation*}
		\lim_{n\to\infty}\uifnal(\utoper^n\chf{\sset}) = 1. 
	\end{equation*}
	Then 
	\begin{equation}\label{en-cor-limit-1}
		\liminf_{n\to\infty} \uifnal(\utoper^n\gamble) \ge \uifnal[M]_{\sset}(\gamble)
	\end{equation}
	for every gamble $\gamble$.
\end{corollary}
\begin{proof}
Let us first show that, for some constant $\alpha$,  $\uifnal[F](\chf{\sset}) \ge \alpha$ implies that 
\begin{equation}\label{eq-minimalIEF-limit-1}
		\liminf_{n\to\infty}\uifnal[F](\utoper^n \gamble) \ge \alpha\uifnal[M]_{\sset}(\gamble)
	\end{equation}
	for every gamble $\gamble$.
By Corollary~\ref{cor-minimalIEF} we have that 
	\begin{equation}\label{eq-prop-alfa-1}
		\liminf_{n\to\infty}\fnal(\utoper^n \gamble) \ge \uifnal[M]_{\sset}(\gamble)
	\end{equation}
	for every (imprecise) expectation functional $\fnal$ such that $\fnal(\chf{\sset}) = 1$.	
	Now suppose that $\uifnal[F](\chf{\sset}) \ge \alpha$. Then there exists some $\fnal\in\ifnal[F]$ such that $\fnal(\chf{\sset}) \ge \alpha$, and by definition we have that $\fnal(\utoper^n \gamble) \le \uifnal[F](\utoper^n\gamble)$ for every $n\in \NN$. Moreover, we have that 
	\begin{equation*}
		\fnal(\utoper ^n \gamble) = \fnal(\chf{\sset})\fnal(\utoper^n\gamble|\sset) + \fnal(\chf{\sset^c}) \fnal(\utoper^n\gamble|\sset^c),
	\end{equation*}
	where $\fnal(\cdot|A)$ is the conditional expectation functional. Hence, 
	\begin{equation*}
		\fnal(\utoper ^n \gamble) \ge \alpha \fnal(\utoper^n\gamble|\sset) = \alpha \fnal(\cdot|\sset)(\utoper^n\gamble).
	\end{equation*}
	However, by \eqref{eq-prop-alfa-1}, $\liminf_{n\to\infty}\fnal(\cdot|\sset)(\utoper^n\gamble)\ge \uifnal[M]_{\sset}(\gamble)$, whence
	\begin{equation*}
		\liminf_{n\to\infty} \uifnal[F](\utoper^n\gamble) \ge \liminf_{n\to\infty} \fnal(\utoper^n\gamble) \ge \alpha \uifnal[M]_{\sset}(\gamble).
	\end{equation*}
	and this proves \eqref{eq-minimalIEF-limit-1}.
Now, for every $\varepsilon>0$ there is some $n\in\NN$ such that $\uifnal(\utoper^n\chf{\sset}) > 1-\varepsilon$. Hence, for every gamble $\gamble$,
	\begin{equation*}
		\liminf_{k\to\infty} \uifnal (\utoper^n\utoper^kf) \ge (1-\varepsilon)\uifnal[M]_{\sset}(\utoper^n\gamble) = (1-\varepsilon)\uifnal[M]_{\sset}(\gamble).
	\end{equation*}
Since the above holds for every $\varepsilon>0$, inequality \eqref{en-cor-limit-1} follows.
\qed
\end{proof}

\begin{corollary}\label{cor-mpc-limit}
Let $B$ be a minimal permanent class. Then a $\toper$-invariant IEF $\ifnal[M]_B$ exists such that for any IEF $\ifnal$ with 
\begin{equation}\label{eq-limit-one}
\lim_{n\to\infty} \uifnal(\utoper^n \chf{B}) = 1. 
\end{equation} 
the inequality
\begin{equation}\label{eq-lmin}
\liminf_{n\to\infty}\uifnal(\utoper^n \gamble) \ge \utoper[M]_B (\gamble)
\end{equation} 
holds for every gamble $\gamble$.
\end{corollary}
\begin{proof}
	By Proposition~\ref{prop-mpc-abs-reg}, $B$ is absorbing and regular with respect to $(\toper_B^r)^k\subseteq \toper^{rk}$. Therefore, by Corollary~\ref{cor-minimalIEF-limit}, IEF $\ifnal[M]_B$ exists such that 	
	\[ \liminf_{n\to\infty} \uifnal(\utoper^{nkr}\gamble) \ge \uifnal[M]_B (\gamble) \]
	for every IEF $\ifnal$ satisfying \eqref{eq-limit-one}, and therefore also for $\ifnal \toper^m$, where $m$ is arbitrary integer. Thus we have	
	\[ \liminf_{n\to\infty} \uifnal(\utoper^{m+nkr}\gamble) = \liminf_{n\to\infty} \uifnal\utoper^m(\utoper^{nkr}\gamble)\ge \uifnal[M]_B (\gamble). \]
It remains to prove that $\ifnal[M]_B$ is $\toper$-invariant. By Corollary~\ref{cor-minimalIEF}, it is $\toper^{rk}$-invariant. Therefore $\ifnal[M]_B\toper = \ifnal[M]_B \toper^{nkr + 1}$. Replacing $\ifnal$ with $\ifnal[M]_B$ in \eqref{eq-lmin}, and using the above identity imply that $\ifnal[M]_B(\toper \gamble) \ge \ifnal[M]_B(\gamble)$ for every gamble, which implies the inclusions $\ifnal[M]_B\subseteq \ifnal[M]_B\toper \subseteq \cdots \subseteq \ifnal[M]_B\toper^{kr} = \ifnal[M]_B$. But then all the inclusions must in fact be equalities, particularly, $\ifnal[M]_B= \ifnal[M]_B\toper$ must hold. 
\qed
\end{proof}

\begin{theorem}\label{thm-main-2}
	Let $\ifnal$ be an IEF on $\states$ such that the limit 
	\begin{equation*}
		\lim_{n\to\infty} \uifnal(\utoper^n\chf{B}) =: \uifnal_{\infty}(\chf{B})
	\end{equation*}
	exists for every minimal permanent class $B$, and $\uifnal_{\infty}(\chf{B})\in \{ 0, 1\}$. Then the limit 
	\begin{equation*}
		\lim_{n\to\infty} \uifnal(\utoper^n\gamble) =: \uifnal_{\infty}(\gamble)
	\end{equation*}
	exists for every gamble $\gamble$. Moreover, $\ifnal_{\infty}$ is the unique $\toper$-invariant IEF such that $\uifnal_{\infty}(\chf{B}) = 1$ for every minimal permanent class $B\subseteq S_{\ifnal}$. 
\end{theorem}
\begin{proof}
	Let $B$ be a minimal permanent class such that 
	\begin{equation}\label{eq-mpc-upex}
		\lim_{n\to\infty} \uifnal(\utoper^n\chf{B}) = 1. 
	\end{equation}	 
	Then by Corollary~\ref{cor-mpc-limit} we have that $\liminf_{n\to\infty}\uifnal(\utoper^n \gamble) \ge \utoper[M]_B (\gamble)$. Let us show that this holds for every minimal permanent class $B\subseteq S_{\ifnal}$. Suppose contrary that there is some $B \subseteq S_{\ifnal}$ such that $\lim_{n\to\infty}\uifnal(\utoper^n\chf{B})\neq 1$. Then, by the assumption, $\lim_{n\to\infty} \uifnal(\utoper^n\chf{B}) = 0$. But, by definition of $S_{\ifnal}$, there must be some $m\ge 0$ such that $\uifnal(\utoper^m\chf{B})>0$, which by Proposition~\ref{prop-monotonicity-psi} implies that $\uifnal(\utoper^{m+n}\chf{B}) \ge \uifnal(\utoper^m\chf{B})$, and therefore the limit as $n\to\infty$ cannot be $0$. This contradiction confirms that \eqref{eq-mpc-upex} holds for every minimal permanent class $B\subseteq S_{\ifnal}$.

	By the above, for every minimal permanent class $B\subseteq S_{\ifnal}$ we have that $\liminf_{n\to\infty}\uifnal(\utoper^n \gamble) \ge \utoper[M]_B (\gamble)$. Let us define 
	\begin{equation*}
	\uifnal[M]_{\ifnal}(\gamble) = \max_{B\subseteq S_{\ifnal}} \uifnal[M]_B(\gamble), 
	\end{equation*}
	which is an IEF with $\uifnal[M]_{\ifnal}(\chf{B}) = 1$ for every $B\subseteq S_{\ifnal}$. Moreover, since 
	\begin{equation*}
	\uifnal[M]_{\ifnal}(\utoper\gamble) = \max_{B\subseteq S_{\ifnal}}\uifnal[M]_B(\utoper \gamble) = \max_{B\subseteq S_{\ifnal}}\uifnal[M]_B(\gamble) = \uifnal[M]_{\ifnal}(\gamble). 
	\end{equation*}
it is a $\toper$-invariant IEF, and by Theorem~\ref{thm-main-1}, it must therefore be the unique $\toper$-invariant IEF with these properties on $S_{\ifnal}$. Hence, 
	\[ \ifnal[M]_{\ifnal} = \lim_{n\to\infty} \ifnal[V]_{\ifnal}\toper^n,  \]
	where $\ifnal[V]_E$ is the vacuous IEF on $S_{\ifnal}$. For the vacuous IEF we have that $\ifnal\subseteq \ifnal[V]_{\ifnal}$, and therefore, for every gamble $\gamble$ we have that 
	\[ \liminf_{n\to\infty} \uifnal(\utoper^n\gamble) \ge \uifnal[M]_{\ifnal}(\gamble) = \lim_{n\to\infty} \uifnal[V]_{\ifnal}(\utoper^n\gamble) \ge \limsup_{n\to\infty} \uifnal(\utoper^n\gamble). \]
	Hence, $\liminf_{n\to\infty} \uifnal(\utoper^n\gamble)= \limsup_{n\to\infty} \uifnal(\utoper^n\gamble) = \lim_{n\to\infty} \uifnal(\utoper^n\gamble) = \uifnal[M]_{\ifnal}(\gamble)$.
\qed
\end{proof}
Let us conclude this section with an example. 
\begin{example}
	Let $\states$ be a set of states partitioned into $\states_1, \states_2$ and $\states_3$. Suppose that an ITO can be decomposed in the following matrix form:
	\begin{equation*}
		\toper = \begin{pmatrix}
			\toper_{11} & 0 & 0 \\ 
			\toper_{21} & \toper_{22} & 0 \\ 
			\toper_{31} & \toper_{32} & \toper_{33} 
		\end{pmatrix}
	\end{equation*}
	where $\utoper_{ii}\chf{\states_i}(\state) = 1$ for every $i=1,2,3$ and every $\state\in\states_i$.  Moreover, assume that $\toper$ induces a regular weak accessibility relation on every $\states_i$. Then $\states_i$ are all minimal permanent classes. If $\utoper_{ij}=0$ whenever $i\neq j$, then we can have that $\uifnal(\chf{\states_i})$ equals $0$ or $1$ independently of each other, with the only condition that $\uifnal(\chf{\states}) = 1$. Therefore we have seven different extremal $\toper$-invariant IEFs. 
	
	However, in the case where $\utoper_{ij}>0$ for some $i\neq j$, $\uifnal(\chf{\states_i})$ clearly implies that $\lim_{n\to\infty} \uifnal(\utoper^n\chf{\states_j})=1$ as well, which restricts the number of possible distinct extremal IEFs. 
	
	In addition to the extremal $\toper$-invariant IEFs, we can obtain the non-extremal ones by forming convex combinations. It is one of the challenges left to explore, whether every non-extremal $\toper$-invariant IEF is equal to such a convex combination. If the answer is positive, this would allow us to characterise all $\toper$-invariant IEFs. 
\end{example}

\section{Conclusions}
The first main contribution of this paper is the introduction of the so-called strong accessibility relation which together with the weak relation, defined earlier by \citet{decooman-2008-a}, allows a detailed analysis of the behaviour of imprecise Markov chains. We have identified the minimal sets of states, named the minimal permanent classes, that correspond to essential communication classes in the classical theory. They are, similarly, the minimal sets of states where the entire probability mass can be concentrated and possibly remains there forever. 

Once the minimal irreducible building blocks had been identified, we could describe and analyse the invariant distributions whose probability mass is entirely concentrated within the minimal permanent classes, and this is also the point of the crucial difference between precise and imprecise theory. While in the precise theory, the entire probability mass can be in one class at the time only, it can 'possibly' be in several classes when the 'imprecise interpretation' is adopted. Therefore the behaviour of an imprecise Markov chain cannot be completely reduced to the behaviour within minimal permanent classes.  Instead, we study the so-called extremal imprecise expectation functionals, which allow the probability mass to be concentrated entirely in one or more permanent classes at the time. 

The extremal imprecise expectation functionals seem to be the minimal objects in the class if imprecise expectation functionals that cannot be further reduced to simpler constituent parts. Our main convergence result in Theorem~\ref{thm-main-1} then shows that when an invariant imprecise distribution is extremal then it is uniquely determined with the corresponding set of possible minimal permanent classes. Moreover, as shown in Theorem~\ref{thm-main-2}, the unique convergence to extremal invariant distributions is ensured already with the convergence of the restrictions to the characteristic functions of minimal permanent classes to either $0$ or $1$. 

Convex combinations of extremal invariant imprecise expectation functionals are the limits of the same combinations of the initial imprecise expectation functionals. Two important challenges that remain subject to further work arise. First is to explore whether every invariant imprecise expectation functional is only a convex combination of a family of extremal ones, and the related question is whether unique convergence when restricted to minimal permanent classes is enough for unique convergence of imprecise Markov chain. Another challenge is the analysis of periodic behaviour in imprecise Markov chains. 

\bibliography{refer}
\end{document}